\documentclass[a4paper] {article}
[15pt]

   \usepackage[top=2.5cm, bottom=2.5cm, left=3.5cm, right= 3.5cm]{geometry}

\makeatletter
\renewcommand*\l@section{\@dottedtocline{1}{1.5em}{2.3em}}
\makeatother

\usepackage{amsfonts}
\usepackage{amssymb}
\usepackage[T1]{fontenc}

\usepackage{tikz}
\usetikzlibrary{calc}

\usepackage{CJK}
\usepackage{amsmath}
 \usepackage{algorithmicx}
 \usepackage[ruled]{algorithm}
 \usepackage{algpseudocode}

\usepackage{amsfonts}
\usepackage{amssymb}
\usepackage{amsthm}
\usepackage{amssymb}
\usepackage{enumerate}
\usepackage[calc]{picture}
\usepackage[all,cmtip]{xy}

\usepackage[mathscr]{eucal}
\usepackage{eqlist}

\usepackage{color}
\usepackage{abstract}
\usepackage[T1]{fontenc}

\theoremstyle{plain}
\newtheorem{theorem}{Theorem}

\newtheorem{lemma}[theorem]{Lemma}

\newtheorem{example}[theorem]{Example}
\newtheorem{corollary}[theorem]{Corollary}
\newtheorem{remark}[theorem]{Remark}

\theoremstyle{definition}
\newtheorem{definition}[theorem]{Definition}

\usepackage{etoolbox}
\newtheoremstyle{myrem}
 {3pt}
 {3pt}
 {\normalsize}
 { }
 {\itshape}
 {:}
 { }
 {}

 \theoremstyle{myrem}
 \appto\remark{\leftskip\parindent}
 \appto\remark{\rightskip\parindent}

\numberwithin{equation}{section}
\numberwithin{theorem}{section}

\begin{document}

\begin{center}
{\Large {\textbf {Discrete Morse Theory on Digraphs}}}
 \vspace{0.58cm}

Yong Lin$^{*}$, Chong Wang$^{\dag}$, Shing-Tung Yau$^{\S}$

\bigskip

\bigskip

    \parbox{24cc}{{\small
{\textbf{Abstract}.}
In this paper, we give a necessary and sufficient condition that  discrete Morse functions on a digraph can be extended to be  Morse functions on its transitive closure, from this we can extend the Morse theory to digraphs by using quasi-isomorphism between path complex and discrete Morse complex, we also prove a general sufficient condition for digraphs that the Morse functions satisfying this necessary and sufficient condition.
}}
 \end{center}

 \vspace{1cc}



\footnotetext[3] {$^\dag$   Corresponding author. }

\section{Introduction}
Digraphs are   generalizations of graphs by assigning a direction or two directions to each edge.  A graph is a digraph where each edge is assigned with two directions.   In 2009, J. Bang-Jensen and G.Z. Gutin \cite{dig} studied digraphs   and gave applications of digraphs in quantum mechanics,  finite automata,  deadlocks of computer processes, etc.  In  2012, A. Grigor'yan, Y. Lin, Y. Muranov and S.T. Yau  \cite{9} initiated the study of path complex on digraphs and defined the path homology of digraphs.  In 2015,  A. Grigor'yan, Y. Lin, Y. Muranov and S.T. Yau  \cite{3, 4}  studied the cohomology of digraphs and graphs by using the path homology theory.   In 2018,  A. Grigor'yan, Y. Muranov, V. Vershinin and S.T. Yau \cite{yau2} generalized the path homology theory of digraphs and constructed the path homology theory of multigraphs and quivers.   

A {\it  digraph} $G$ is a pair $(V,E)$ where $V$ is a set  and $E$ is a subset of $V\times V$. The elements of $V$ are called {\it vertices} and $V$ is called the {\it vertex set}. For any vertices $u, v\in V$,  if  $(u, v)\in E$, then $(u, v)$ is called a {\it directed edge}, and is denoted as $u\to v$. For any vertex $v\in V(G)$, the number of directed edges starting from $v$ is called the out-degree of $v$, and the number of directed edges ending at $v$ is called the  in-degree of $v$. The sum of in-degree and out-degree is called the degree of $v$, denoted as $D(v)$. The transitive closure $\overline{G}$  of a digraph $G$ is the smallest digraph containing $G$ such that for any two directed edges $u\to v$ and $v\to w$ of $\overline{G}$,  there is a directed edge $u\to w$ of $\overline{G}$.


Let $G$ be a digraph and $V$ be the vertex set of $G$.   For each $n\geq 0$, an {\it  elementary $n$-path} (or {\it $n$-path} for short) on $V$ is a sequence $v_0v_1\cdots v_n$ of vertices in $V$ where $v_0$, $v_1$, $\ldots$, $v_n\in V$. Here the vertices $v_0$, $v_1$, $\ldots$, $v_n$ are not required to be distinct.  An {\it  allowed elementary $n$-path} on $G$ is a $n$-path $v_0 v_1\ldots v_n$ on $V$ such that for each $i\geq 1$, $v_{i-1}\to v_i$ is a directed edge of $G$ and  $v_{i-1}\not=v_i$ for each $1\leq i \leq n$. Let $\Lambda_n(V)$ be the free $R$-module consisting of all the formal linear combinations (with coefficients in a commutative ring $R$ with unit) of the $n$-paths on $V$. Let $P_n(G)$ be the free $R$-module consisting of all the formal linear combinations of allowed elementary $n$-paths on $G$. Then $P_n(G)$ is a   sub-$R$-module of $\Lambda_n(V)$.

The boundary map $\partial_n: \Lambda_n(V)\longrightarrow \Lambda_{n-1}(V)$ is defined as
by letting
\begin{eqnarray*}
\partial_n(v_0v_1\ldots v_n)=\sum_{i=0}^n (-1)^i d_i(v_0v_1\ldots v_n)
\end{eqnarray*}
where $d_i$ is the face map given by
\begin{eqnarray*}
d_i(v_0v_1\ldots v_n)= v_0v_1\ldots\hat{v_i}\ldots v_n.
\end{eqnarray*}
Note that $\partial_n$ is an $R$-linear map from $\Lambda_n(V)$ to $\Lambda_{n-1}(V)$ satisfying $\partial_{n}\partial_{n+1}=0$ for each $n\geq 0$ (cf. \cite{9, 10, 3, 4, 2, yau2, 1}).  Hence $\{\Lambda_n(V),\partial_n\}_{n\geq 0}$ is a chain complex.
We define
\begin{eqnarray*}
\Omega_n(G)&=& P_n(G) \cap (\partial_n)^{-1} P_{n-1}(G),\\
\Gamma_n(G)&=& P_n(G)+ \partial_{n+1} P_{n+1}(G).
\end{eqnarray*}
Then as graded $R$-modules,
\begin{eqnarray*}
\Omega_*(G) \subseteq  P_*(G)\subseteq  \Gamma_*(G) \subseteq \Lambda_*(V).
\end{eqnarray*}
And as chain complexes,
\begin{eqnarray*}
\{\Omega_n(G), \partial_n\mid_{\Omega_n(G)} \}_{n\geq 0} \subseteq  \{\Gamma_n(G), \partial_n\mid_{\Gamma_n(G)} \}_{n\geq 0}\subseteq \{\Lambda_n(V),\partial_n\}_{n\geq 0}.
\end{eqnarray*}
By \cite[Proposition~2.4]{h1}, the canonical inclusion
\begin{eqnarray*}
\iota:  \Omega_n(G)  \longrightarrow  \Gamma_n(G), ~~~ n\geq 0
\end{eqnarray*}
of chain complexes  induces an isomorphism between the homology groups
\begin{eqnarray*}
\iota_*: H_m(\{\Omega_n(G), \partial_n\mid_{\Omega_n(G)} \}_{n\geq 0})\overset{\cong}{\longrightarrow} H_m(\{\Gamma_n(G), \partial_n\mid_{\Gamma_n(G)} \}_{n\geq 0}), ~~~ m\geq 0.
\end{eqnarray*}
This isomorphism gives  the {\it path homology} of $G$.   We denote the path homology  defined in this way  as $H_n(G;R)$ ($n\geq 0$), or  simply $H_n(G)$  if there is no danger of confusion.

Morse theory originated from the study of homology groups and cell structure of smooth manifolds. In the 1990s, the discrete Morse theory for cell complexes and simplicial complexes was given (cf. \cite{forman1, forman2, forman3, witten}). In recent years, the discrete Morse theory of cell complexes and simplicial complexes has been applied to graphs, and the discrete Morse theory of graphs has been studied (cf. \cite{ayala1, ayala2, ayala3, ayala4}). Discrete Morse theory can greatly reduce the number of cells and simplices,  simplify the calculation of homology groups, and can be applied to topological data analysis (cf. \cite{lewiner, nanda, kannan}, etc). In these study, people use clique as flag complexes on graph which is a similarity of simplicial complexes.


In this paper,  we further study the discrete Morse theory for digraphs, try to make use of the discrete Morse theory to greatly reduce the initial information which does not affect the path homology groups of digraphs,  so as to simplify the calculation. Not analogue to the flag complex, the sub complex of a path complex is not necessary a path complex. So we need to extend the path complex to its transitive closure and prove that path homology is invariant under this extension.

Let $G$ be a digraph and $f: V(G)\longrightarrow [0,+\infty)$  a discrete Morse function on $G$ as defined in \cite{wang} and Definition 2.1 in next section.  Consider the following condition
\begin{quote}
($*$). for each vertex $v\in V(G)$,  there exists at most one zero point of $f$ in all allowed elementary paths starting or ending at  $v$.
\end{quote}
Then $f$ can be extended to be  a Morse function $\overline{f}$ on the transitive closure $\overline{G}$ of $G$ such that $\overline{f}(v)=f(v)$ for each vertex $v\in E(G)$ if and only if $f$ satisfies Condition ($*$). This is proved in Theorem~\ref{th-666}.

For $n\geq 0$, define an $R$-linear map $\mathrm{grad}f: P_n(G)\rightarrow P_{n+1}(G)$ such that for  an allowed elementary n-path $\alpha$ on $G$, if there exists an allowed elementary $(n + 1)$-path $\gamma$ satisfying $\gamma>\alpha$ and $f(\gamma)=f(\alpha)$, set
\begin{eqnarray*}
(\mathrm{grad}f)(\alpha)=-\langle\partial\gamma,\alpha\rangle\gamma
\end{eqnarray*}
and otherwise $(\mathrm{grad}f)(\alpha)=0$. We call $\mathrm{grad}f$ the (algebraic) discrete gradient vector field of $f$, denoted as $V_f$. Here $\langle, \rangle$ is  the  inner product in $\Lambda_n(V)$  (with respect to which the $n$-paths are orthonormal).

Let $\overline{V}=\mathrm{grad}\overline{f}$ be the discrete gradient vector field on $\overline{G}$. By \cite[Definition~6.2]{forman1}), define the discrete gradient flow of $\overline{G}$ as
\begin{eqnarray*}
\overline{\Phi}=\mathrm{Id}+\partial \overline{V}+\overline{V}\partial,
\end{eqnarray*}
which is an $R$-linear map
\begin{eqnarray*}
\overline{\Phi}: P_n(\overline{G})\longrightarrow P_n(\overline{G}), ~~~~n\geq 0.
\end{eqnarray*}
Denote the stabilization map of $\overline{\Phi}$ as $\overline{\Phi}^{\infty}$. Let $\mathrm{Crit}_n(G)$ be the vector space consisting of all the formal linear combinations of critical $n$-paths (see Definition~\ref{def-1.1}) on $G$. Suppose $\Omega_*(G)$ is $\overline{V}$-invariant, that is, $\overline{V}(\Omega_*(G))\subseteq \Omega_*(G)$. Then in Theorem~\ref{th-8.3}, we prove that
\begin{eqnarray*}
H_m(G)\cong  H_m(\{\Omega_n(G)\cap\overline{\Phi}^{\infty}(\mathrm{Crit}_n(\overline{G})),  \partial_n\}_{n\geq 0}),~~~ m\geq 0.
\end{eqnarray*}

Finally, in Theorem~\ref{th-33}, it is proved  that Condition ($*$) is not very harsh. In fact, we can define Morse functions satisfying Condition ($*$) on quite general digraphs, so Theorem~\ref{th-8.3} can be used  to simplify the calculation of their path homology groups. We give some examples to illustrate.

\smallskip

\section{Discrete Morse Theory for Digraphs}

In this section,   in Theorem~\ref{th-666}, we study the  extendability of a Morse function $f$ on digraph $G$ to its transitive closure $\overline{G}$. In Theorem~\ref{th-6.11}, we  prove a quasi-isomorphism of chain complexes  and give the discrete Morse theory for digraphs in Theorem~\ref{th-8.3}. 

\smallskip

\subsection{Definitions and Some Properties}

Let $G$ be a digraph. For any allowed elementary paths $\gamma$ and $\gamma'$, if $\gamma'$ can be obtained from $\gamma$ by removing some vertices, then we write $\gamma'<\gamma$ or $\gamma>\gamma'$.

\begin{definition}(cf. \cite{wang})\label{def-1.1}
 A  map $f: V(G)\longrightarrow [0,+\infty)$ is called a  {\it discrete Morse function} on $G$, if for any allowed elementary path $\alpha=v_0v_1\cdots v_n$ on $G$, both of the followings hold:
 \begin{quote}
\begin{enumerate}[(i).]
\item
$\#\Big\{\gamma^{(n+1)}>\alpha^{(n)}\mid f(\gamma)=f(\alpha)\Big\}\leq 1$;
\item
$\#\Big\{\beta^{(n-1)}<\alpha^{(n)}\mid f(\beta)=f(\alpha)\Big\}\leq 1$.
\end{enumerate}
\end{quote}
where
\begin{eqnarray}\label{eq-1.1}
f(\alpha)=f(v_0v_1\cdots v_n)= \sum_{i=0}^n f(v_i).
\end{eqnarray}
For an allowed elementary path $\alpha$, if in both (i) and (ii), the inequalities hold strictly, then $\alpha$ is called {\it critical}.
\end{definition}

\begin{definition}(cf. \cite[Definition~0.6]{witten})\label{def-witten}
A function $f: V(G)\longrightarrow [0,+\infty)$ is called a  {\it discrete Witten-Morse function} on $G$ if, for any allowed elementary path $\alpha$,
\begin{quote}
\begin{enumerate}[(i)]
\item
$f(\alpha)<\mathrm{average}\{f(\gamma_1), f(\gamma_2)\}$ where $\gamma_1>\alpha$ and $\gamma_2>\alpha$;
\item
$f(\alpha)>\mathrm{average}\{f(\beta_1), f(\beta_2)\}$ where $\beta_1<\alpha$ and $\beta_2<\alpha$.
\end{enumerate}
\end{quote}
\end{definition}
Note that each Witten-Morse function is, in fact, a Morse function.

\begin{definition}(cf. \cite[Definition~0.7]{witten})\label{def-flat}
A discrete Witten-Morse funtion is {\it flat} if, for any allowed elementary path $\alpha$
\begin{quote}
\begin{enumerate}[(i)]
\item
$f(\alpha)\leq \mathrm{min}\{f(\gamma_1), f(\gamma_2)\}$ where $\gamma_1>\alpha$ and $\gamma_2>\alpha$;
\item
$f(\alpha)\geq \mathrm{max}\{f(\beta_1), f(\beta_2)\}$ where $\beta_1<\alpha$ and $\beta_2<\alpha$.
\end{enumerate}
\end{quote}
\end{definition}

By (\ref{eq-1.1}) and Definitions~\ref{def-flat}, it follows that each discrete Morse function on a digraph is a discrete flat Witten-Morse function.

A directed loop on $G$ is an allowed elementary path $v_0 v_1 \cdots v_n v_0$, $n \geq 1$.
\begin{lemma}\label{le-1.14}
Let $G$ be a digraph and $f$  a discrete Morse function on $G$. Let $\alpha=v_0v_1\cdots v_n v_0$ be a directed loop. Then for each $0\leq i\leq n$, $f(v_i)>0$.
\end{lemma}
\begin{proof}
 Suppose to the contrary, $f(v_i)=0$ for some $i$. Let $\alpha'=v_iv_{i+1}\cdots v_n v_0\cdots v_{i-1}v_i$, $\beta_1=v_iv_{i+1}\cdots v_n v_0\cdots v_{i-1}$ and $\beta_2=v_{i+1}\cdots v_n v_0\cdots v_{i-1}v_i$ where $v_{i-1}=v_n$ for $i=0$ and $v_{n+1}=v_0$ for $i=n$. Then $f(\alpha')=f(\beta_1)=f(\beta_2)$. This contradicts Definition~\ref{def-1.1} (ii). The lemma follows.
\end{proof}

\begin{lemma}\label{le-1.13}
Let $G$ be a digraph and $f$  a discrete Morse function on $G$. Then for any allowed elementary path $\alpha=v_0\cdots v_n$ in $G$, there exists at most one index $i$ such that $f(v_i)=0$.
\end{lemma}
\begin{proof}
 Suppose to the contrary, there are two  indices $i$ and $j$ such that $f(v_i)=f(v_j)=0$ ($i\not=j$). Without loss of generality, $i<j$. Since $\alpha$ is allowed, $v_i\not=v_{i+1}$. 
 Let $\alpha'=v_i\cdots v_j$, $\beta_1=v_i\cdots v_{j-1}$ and $\beta_2=v_{i+1}\cdots v_j$. Then $\beta_1\not=\beta_2$ and  $f(\alpha')=f(\beta_1)=f(\beta_2)$. This contradicts Definition~\ref{def-1.1} (ii). Hence the lemma follows.
\end{proof}

\subsection{Extension of Morse Functions on Digraphs}

\begin{definition}\cite[Section~2.3]{dig}\label{trans}
 A digraph $G$ is called {\it transitive}, if for any two directed edges $u\to v$ and $v\to w$ of $G$, there is a directed edge $u\to w$ of $G$.
\end{definition}
\begin{remark}\label{re-3.1.1}
The digraph $G$ is transitive if and only if  $P(G)$ is  {\it perfect} (\cite[Definition~3.4]{9}).
\end{remark}

The next lemma is straight-forward to verify.

\begin{lemma}\cite[Section~2.3]{dig}\label{le-6.2.2}
For any digraph $G$, there exists a digraph $\overline G$ such that
\begin{enumerate}[(i).]
\item
each directed edge of $G$ is a directed edge of $\overline G$;
\item
$\overline G$ is transitive;
\item
any digraph $G'$ satisfying (i) and (ii), $\overline G$ is contained in $G'$.
\end{enumerate}
We call $\overline G$ the {\it  transitive closure} of $G$. A digraph $G$ is transitive if and only if $\overline G=G$.
\qed
\end{lemma}
\smallskip
\begin{remark}\label{re-11}
For any directed edge $u\to v$ in $E(\overline{G})\setminus E(G)$, there exists a sequence of vertices $w_1w_2\cdots w_k$ ($k\geq 1$) in $V(G)$ such that $uw_1\cdots w_k v$ is an allowed elementary path on $G$. For example,

\begin{figure}[!htbp]
 \begin{center}
\begin{tikzpicture}[line width=1.5pt]

\coordinate [label=left:$v_0$]    (A) at (2,2);
 \coordinate [label=left:$v_1$]   (B) at (3,3);
 \coordinate  [label=right:$v_2$]   (C) at (4,3);
\coordinate  [label=right:$v_3$]   (D) at (5,2);
\coordinate  [label=right:$v_4$]   (E) at (4,1);
 \coordinate [label=left:$v_5$]   (F) at (3,1);

\coordinate[label=left:$G$:] (H) at (1,2);
  \draw [line width=1.5pt]  (A) -- (B);
   \draw [line width=1.5pt]  (B) -- (C);
      \draw [line width=1.5pt]  (C) -- (D);
      \draw [line width=1.5pt]  (D) -- (E);
            \draw [line width=1.5pt]  (A) -- (F);
                 \draw [line width=1.5pt]  (F) -- (E);

 \draw[->] (2,2) -- (2.5,2.5);
 \draw[->] (2,2) -- (2.5,1.5);

\draw[->] (3,3) -- (3.5,3);

\draw[->] (4,3) -- (4.5,2.5);

\draw[->] (3,1) -- (3.5,1);

\draw[->] (5,2) -- (4.5,1.5);

\fill (2,2)  circle (2.5pt) (3,3)  circle (2.5pt)  (4,3) circle (2.5pt) (5,2) circle (2.5pt) (4,1) circle (2.5pt) (3,1)  circle (2.5pt);

\coordinate [label=left:$v_0$]    (A) at (2+6,2);
 \coordinate [label=left:$v_1$]   (B) at (3+6,3);
 \coordinate  [label=right:$v_2$]   (C) at (4+6,3);
\coordinate  [label=right:$v_3$]   (D) at (5+6,2);
\coordinate  [label=right:$v_4$]   (E) at (4+6,1);
 \coordinate [label=left:$v_5$]   (F) at (3+6,1);

\coordinate[label=left:$\overline{G}$:] (I) at (1+6,2);
  \draw [line width=1.5pt]  (A) -- (B);
   \draw [line width=1.5pt] (A) -- (C);
    \draw [line width=1.5pt]  (A) -- (D);
    \draw [line width=1.5pt]  (A) -- (E);

   \draw [line width=1.5pt]  (B) -- (C);
   \draw [line width=1.5pt]  (B) -- (D);
    \draw [line width=1.5pt]  (B) -- (E);

      \draw [line width=1.5pt]  (C) -- (D);
       \draw [line width=1.5pt]  (C) -- (E);

      \draw [line width=1.5pt]  (D) -- (E);
            \draw [line width=1.5pt]  (A) -- (F);
                 \draw [line width=1.5pt]  (F) -- (E);

 \draw[->] (2+6,2) -- (2.5+6,2.5);
 \draw[->] (2+6,2) -- (2.5+6,1.5);
 \draw[->] (2+6,2) -- (10.5,2);
 \draw[->] (2+6,2) -- (9.8,2.9);
  \draw[->] (2+6,2) -- (9.8,1.1);

\draw[->] (3+6,3) -- (3.5+6,3);
\draw[->] (3+6,3) -- (10.6,2.2);
\draw[->] (3+6,3) -- (9.7,1.6);

\draw[->] (4+6,3) -- (4.5+6,2.5);
\draw[->] (4+6,3) -- (10,1.5);

\draw[->] (3+6,1) -- (3.5+6,1);

\draw[->] (5+6,2) -- (4.5+6,1.5);

\fill (2+6,2)  circle (2.5pt) (3+6,3)  circle (2.5pt)  (4+6,3) circle (2.5pt) (5+6,2) circle (2.5pt) (4+6,1) circle (2.5pt) (3+6,1)  circle (2.5pt);
 \end{tikzpicture}
\end{center}
\caption{Remark~\ref{re-11}.}
\end{figure}
\begin{eqnarray*}
E(\overline{G})\setminus E(G)=\{v_0\to v_2, v_0\to v_3, v_0\to v_4, v_1\to v_3, v_1\to v_4, v_2\to v_4\},
\end{eqnarray*}
in which there exists an allowed elementary path on $G$  for each edge in $E(\overline{G})\setminus E(G)$. Specifically,
\begin{eqnarray*}
&&v_0\to v_2\in E(\overline{G})\setminus E(G)  \text{ corresponds to } v_0v_1v_2 \in P(G);\\
&&v_0\to v_3\in E(\overline{G})\setminus E(G)  \text{ corresponds to } v_0v_1v_2v_3 \in P(G);\\
&&v_0\to v_4\in E(\overline{G})\setminus E(G)  \text{ corresponds to } v_0v_1v_2v_3v_4 \in P(G);\\
&&v_1\to v_3\in E(\overline{G})\setminus E(G)  \text{ corresponds to } v_1v_2v_3 \in P(G);\\
&&v_1\to v_4\in E(\overline{G})\setminus E(G)  \text{ corresponds to } v_1v_2v_3v_4 \in P(G);\\
&&v_2\to v_4\in E(\overline{G})\setminus E(G)  \text{ corresponds to } v_2v_3v_4 \in P(G).
\end{eqnarray*}
Moreover, let $\alpha=v_0\cdots v_4\in P(G)$, $v_i\to v_j$ ($0\leq i<j\leq 4$) are directed edges in $E(\overline{G})$.
\end{remark}

\smallskip

 Let $G$ be a digraph and $f$ a discrete Morse function on $G$. Then $f$ may {\bf not} be extendable to be a discrete Morse function on $\overline{G}$.

\begin{example}\label{ex-1.1}
Let $G$ be a digraph with the set of vertices $V(G)=\{v_0,v_1,v_2, v_3\}$ and the set of directed edges $E(G)=\{v_0\to v_3, v_1\to v_2, v_2\to v_3\}$. Then
 \begin{eqnarray*}
 P(G)=R\{v_0, v_1, v_2, v_3, v_0v_3, v_1v_2, v_2v_3, v_1v_2v_3\}
 \end{eqnarray*}
 and the transitive closure of $G$ is a digraph $\overline{G}$  with $ V(\overline{G})=V(G)$  and
  \begin{eqnarray*}
  E(\overline{G})=E(G)\cup \{v_1\to v_3\}.
\end{eqnarray*}


\begin{figure}[!htbp]
 \begin{center}
\begin{tikzpicture}[line width=1.5pt]

\coordinate [label=right:$v_1$]    (A) at (4,4);
 \coordinate [label=left:$v_2$]   (B) at (2,4);
 \coordinate  [label=left:$v_3$]   (C) at (2,2);
\coordinate  [label=right:$v_0$]   (D) at (4,2);

\coordinate[label=left:$G$:] (G) at (1,3);
   \draw [line width=1.5pt]  (A) -- (B);
    \draw [line width=1.5pt]  (B) -- (C);
      \draw [line width=1.5pt]  (D) -- (C);
 \draw[->] (4,4) -- (2.5,4);

\draw[->] (2,4) -- (2,2.5);

\draw[->] (4,2) -- (2.5,2);

\fill (4,4) circle (2.5pt) (2,4) circle (2.5pt) (2,2) circle (2.5pt) (4,2) circle (2.5pt);

\coordinate [label=right:$v_1$]    (A) at (4+6,4);
 \coordinate [label=left:$v_2$]   (B) at (2+6,4);
 \coordinate  [label=left:$v_3$]   (C) at (2+6,2);
\coordinate  [label=right:$v_0$]   (D) at (4+6,2);

\coordinate[label=left:$\overline{G}$:] (H) at (1+6,3);
   \draw [line width=1.5pt]  (A) -- (B);
    \draw [line width=1.5pt]  (B) -- (C);
      \draw [line width=1.5pt]  (D) -- (C);
      \draw [line width=1.5pt]  (A) -- (C);
 \draw[->] (4+6,4) -- (2.5+6,4);
\draw[->] (4+6,4) -- (9,3);
\draw[->] (2+6,4) -- (2+6,2.5);

\draw[->] (4+6,2) -- (2.5+6,2);

\fill (4+6,4) circle (2.5pt) (2+6,4) circle (2.5pt) (2+6,2) circle (2.5pt) (4+6,2) circle (2.5pt);

 \end{tikzpicture}
\end{center}

\caption{Example~\ref{ex-1.1}.}
\end{figure}
Let $f$ be a function on $V(G)$ with $f(v_1)=f(v_0)=0$ and $f(v_2)>0, f(v_3)>0$. Then by Definition~\ref{def-1.1},  $f$ is a discrete Morse function on $G$. However, since $f(v_3)=f(v_0v_3)=f(v_1v_3)$, $f$ is not a discrete Morse function on $\overline{G}$. Hence, there does not exist any discrete Morse
function $\overline{f}$ on $\overline{G}$ such that the restriction of $\overline{f}$ to $G$ equals $f$.
\end{example}
To give a condition for extendability of a Morse function $f$ on $G$ to its transitive closure, we study the property of the discrete Morse function on transitive digraph in the following lemma first. 

\begin{lemma}\label{le-555}
Let $G$ be a transitive digraph and $f: V(G)\longrightarrow [0,+\infty)$ a discrete Morse function on $G$. Then for any vertex $v$ with $D(v)\geq 2$, there exists at most one vertex $w\not=v$ such that $f(w)=0$ and $v\to w$ or $w\to v$ is a directed edge in $E(G)$.
\end{lemma}
\begin{proof}
Suppose to the contrary, there are   three cases to be considered.

{\sc Case~1}. There are two vertices $w_1, w_2\in V(G)$ such that $f(w_1)=f(w_2)=0$ and $v\to w_1, w_2\to v \in E(G)$.

{\sc Case~2}. There are two vertices $w_1, w_2\in V(G)$ such that $f(w_1)=f(w_2)=0$ and $w_1\to v,  w_2 \to v \in E(G)$.

{\sc Case~3}. There are two vertices $w_1, w_2\in V(G)$ such that $f(w_1)=f(w_2)=0$ and $v\to w_1,  v\to w_2 \in E(G)$.

Without loss of generality, we only give the proof of Case 1. Let $\alpha=v$, $\gamma_1=vw_1$ and $\gamma_2=w_2v$. Then $f(\alpha)=f(\gamma_1)=f(\gamma_2)$ and $\gamma_1>\alpha, \gamma_2>\alpha$.  This contradicts that $f$ is a discrete Morse function on $G$. The lemma follows.
\end{proof}



Next, we give a necessary and sufficient condition for a discrete Morse function on a digraph to be extended to its transitive closure in the following theorem.

\begin{theorem}\label{th-666}
Let $G$ be a digraph and $f: V(G)\longrightarrow [0,+\infty)$ a discrete Morse function on $G$.   Then $f$ can be extended to be a Morse function $\overline{f}$ on $\overline{G}$ such that $\overline{f}(v)=f(v)$ for each vertex $v\in E(G)$ if and only if Condition ($*$) is satisfied.
\end{theorem}
\begin{proof}
Suppose the discrete Morse function $f$ on $G$ can be extended to be a discrete Morse function on $\overline{G}$. Let $\alpha=v_0\cdots v_n$ be an allowed elementary path on $G$ with $v_0=v$ or $v_n=v$. By Lemma~\ref{le-6.2.2}, for each $v_k$ ($0\leq k\leq n$) such that $v_k\not=v$,  $v\to v_k$ or $v_k\to v$ is a directed edge in $E(\overline{G})$. Then by Lemma~\ref{le-1.13} and Lemma~\ref{le-555}, $f$ satisfies Condition ($*$) obviously.


On the other hand, suppose $f$ satisfies Condition ($*$). Let $\alpha=v_0v_1\cdots v_n$ be an allowed elementary path on $\overline{G}$. Firstly, we assert that  there exists at most one index $i$ ($0\leq i\leq n$) such that  ${f}(v_i)=0$ and  $d_i(\alpha)$ is an allowed elementary $(n-1)$-path on $\overline{G}$.

{\sc Case~1}.   $\alpha\in P(G)$. Then by Lemma~\ref{le-1.13}, the assertion follows. 

{\sc Case~2}.  $\alpha\not\in P(G)$.   Then there exists a directed edge $v_{i}v_{i+1}\in E(\overline{G})\setminus E(G)$ for some $0\leq i\leq {n-1}$. Hence, by Remark~\ref{re-11}, there exists an allowed elementary path $\alpha'=v_0\cdots v_i w_1\cdots w_k v_{i+1}\cdots v_n$ on $G$ with $k\geq 1$ and $w_1, \cdots, w_k\in V(G)$. Hence, by  Lemma~\ref{le-1.13}, the assertion follows.

Secondly, we assert that there exists at most one allowed  elementary $(n+1)$-path $\alpha'=v_0\cdots v_juv_{j+1}\cdots v_n$ on $\overline{G}$ with $f(u)=0$, $u\in V(G)$. Suppose to the contrary, we consider the following two cases.

{\sc Case~3}.  
$\alpha''=v_0\cdots v_i wv_{i+1}\cdots v_n$  is another  allowed elementary $(n+1)$-path on $\overline{G}$ with $f(w)=0$ and $i\not=j$. Without loss of generality, $j>i$. Then by Remark~\ref{re-11}, there is an allowed elementary path with $w$ as the  starting point and $u$ the ending point on $G$. This contradicts Lemma~\ref{le-1.13}.

{\sc Case~4}.  
$\alpha''=v_0\cdots v_i wv_{i+1}\cdots v_n$  is another  allowed elementary $(n+1)$-path on $\overline{G}$ with $f(w)=0$ and $i=j$. Then $u\not=w$. Hence, by Remark~\ref{re-11}, there are two allowed elementary paths on $G$ with $v_i$ as the starting  point and $u, w$ as the ending point respectively, or with $u, w$ as  the starting point respectively  and $v_{i+1}$ as the ending point. This contradicts Condition ($*$).

Summarising above cases,  $f$ can be extended to  be a Morse function $\overline{f}$ on $\overline{G}$.

Therefore, the theorem is proved.
\end{proof}


\subsection{Quasi-isomorphism, Discrete Morse Theory for Digraphs}
Let $f$ be a discrete Morse function on  digraph $G$ and $\overline{f}$ the extension of $f$ on the transitive closure $\overline{G}$ of $G$. Let $V_f=\mathrm{grad}f$ and $\overline{V}=\mathrm{grad}\overline{f}$ be discrete gradient vector fields on $G$ and $\overline{G}$  respectively. Generally, $\overline{V}\mid_{P(G)}\not=V_f$.

\begin{example}\label{ex-55}
Let $G$ be a digraph with $V(G)=\{v_0,v_1,v_2, v_3\}$ and $E(G)=\{v_0\to v_1, v_1\to v_2, v_2\to v_3, v_0\to v_3\}$. Then  the transitive closure of $G$ is a digraph $\overline{G}$  with  $V(\overline{G})=V(G)$  and
 \begin{eqnarray*}
  E(\overline{G})=E(G)\cup \{v_0\to v_2, v_1\to v_3\}.
\end{eqnarray*}


\begin{figure}[!htbp]
 \begin{center}
\begin{tikzpicture}[line width=1.5pt]

\coordinate [label=left:$v_0$]    (A) at (2,2);
 \coordinate [label=left:$v_1$]   (B) at (2,4);
 \coordinate  [label=right:$v_2$]   (C) at (4,4);
 \coordinate  [label=right:$v_3$]   (D) at (4,2);

\coordinate[label=left:$G$:] (G) at (1,3);
   \draw [line width=1.5pt]  (A) -- (B);
    \draw [line width=1.5pt]  (B) -- (C);
    \draw [line width=1.5pt]  (C) -- (D);
      \draw [line width=1.5pt]  (A) -- (D);
 \draw[->] (2,2) -- (2,3);

\draw[->] (2,4) -- (3,4);

\draw[->] (4,4) -- (4,3);
\draw[->] (2,2) -- (3,2);

\fill (2,2) circle (2.5pt) (2,4) circle (2.5pt) (4,4) circle (2.5pt) (4,2) circle (2.5pt);

\coordinate [label=left:$v_0$]    (A) at (2+6,2);
 \coordinate [label=left:$v_1$]   (B) at (2+6,4);
 \coordinate  [label=right:$v_2$]   (C) at (4+6,4);
\coordinate  [label=right:$v_3$]   (D) at (4+6,2);

\coordinate[label=left:$\overline{G}$:] (H) at (1+6,3);
   \draw [line width=1.5pt]  (A) -- (B);
       \draw [line width=1.5pt]  (B) -- (C);
    \draw [line width=1.5pt]  (C) -- (D);
      \draw [line width=1.5pt]  (A) -- (D);
        \draw [line width=1.5pt]  (B) -- (D);
          \draw [line width=1.5pt]  (A) -- (C);
 \draw[->] (2+6,2) -- (2+6,3);

\draw[->] (2+6,4) -- (3+6,4);

\draw[->] (4+6,4) -- (4+6,3);
\draw[->] (2+6,2) -- (3+6,2);
\draw[->] (2+6,4) -- (9.5,2.5);
\draw[->] (2+6,2) -- (9.5,3.5);
\fill (2+6,2) circle (2.5pt) (2+6,4) circle (2.5pt) (4+6,4) circle (2.5pt) (4+6,2) circle (2.5pt);

 \end{tikzpicture}
\end{center}

\caption{Example~\ref{ex-55}.}
\end{figure}
Let $f$ be a function on $V(G)$ with $f(v_2)=0$, $f(v_0)>0, f(v_1)>0$, and $f(3)>0$. Then $f$ is a discrete Morse function on $G$ satisfying Condition ($*$).  By Theorem~\ref{th-666}, $f$ can be extended to be a Morse function $\overline{f}$ on $\overline{G}$ such that $\overline{f}(v)=f(v)$ for each vertex $v\in V(G)$.

Let $\alpha=v_0v_3\in P(G)\subseteq P(\overline{G})$. Since there exists no allowed elementary path $\gamma\in P(G)$  such that $\gamma>\alpha$ and $f(\gamma)=f(\alpha)$. Then $V_f(\alpha)=0$.  Since  $\overline{f}(v_0v_2v_3)=\overline{f}(\alpha)$ and $\partial(v_0v_2v_3)=v_2v_3-v_0v_3+v_0v_2$,  then $\overline{V}(\alpha)=v_0v_2v_3\in P(\overline{G})$. Hence the restriction of $\overline{V}$ on P(G) may not be $V_f$.
\end{example}
\smallskip
Denote
\begin{eqnarray*}
\overline{\Phi}=\mathrm{Id}+\partial \overline{V}+\overline{V}\partial,
\end{eqnarray*}
as the discrete gradient flow of $\overline{G}$. Similar with the proof of  \cite[Theorem~6.3, Theorem~6.4]{forman1}, the main properties of $\overline{V}$ and $\overline{\Phi}$  are contained in the following theorem.

\begin{theorem}\label{th-6.3}(cf. \cite[Theorem~6.3, Theorem~6.4]{forman1})
\begin{enumerate}[(i).]
\item
$\overline{V}\circ {\overline{V}}=0$;
\item
$\sharp\{\beta^{(n-1)}\mid{\overline{V}(\beta)=\pm{\alpha}}\}\leq 1$ for any $\alpha\in P(\overline{G})$;
\item
$\alpha$ is critical $\Longleftrightarrow \{\alpha\not \in \mathrm{Image}(\overline{V})~\text{and}~ \overline{V}(\alpha)=0\}$ for any $\alpha\in P(\overline{G})$;
\item
$\overline{\Phi}\partial=\partial\overline{\Phi}$.
\end{enumerate}
\end{theorem}
\begin{proof}
The simple proof of the theorem is as follows.
\begin{enumerate}[(i).]
\item
For any allowed elementary path $\beta$ on $\overline{G}$, if $\overline{V}(\beta)=\pm\alpha$, then $\beta<\alpha$ and $\overline{f}(\beta)=\overline{f}(\alpha)$. By Lemma~\ref{le-1.13}, there is no allowed elementary path $\gamma$ on $\overline{G}$ such that $\gamma>\alpha$ and $\overline{f}(\gamma)=\overline{f}(\alpha)$. Hence $\overline{V}\circ\overline{V}(\beta)=0$. (i) is proved.
\item
If $\overline{V}(\beta)=\pm{\alpha}$, then $\beta<\alpha$ and $\overline{f}(\beta)=\overline{f}(\alpha)$. Hence by Definition~\ref{def-1.1}(ii), (ii) is proved.
\item
By Definition~\ref{def-1.1},
\begin{eqnarray*}
\alpha\text{ is critical}&\Longleftrightarrow& \text{for any allowed elementary path }\gamma>\alpha, \overline{f}(\gamma)>\overline{f}(\alpha)\text{ and }\\
 &&\text{for any allowed elementary path } \beta<\alpha, \overline{f}(\beta)<\overline{f}(\alpha).
 \end{eqnarray*}
That is equivalent to there is no  allowed elementary path $\gamma$  on $\overline{G}$ such that $\gamma>\alpha$ and $\overline{f}(\gamma)=\overline{f}(\alpha)$, and no  allowed elementary path $\beta$ on $\overline{G}$ such that $\beta<\alpha$  and $\overline{f}(\beta)=\overline{f}(\alpha)$, which implies (iii).
\item
\begin{eqnarray*}
\partial \overline{\Phi}&=&\partial(\mathrm{Id}+\partial \overline{V}+\overline{V}\partial)=\partial+\partial\overline{V}\partial,\\
\overline{\Phi}\partial&=&(\mathrm{Id}+\partial \overline{V}+\overline{V}\partial)\partial=\partial+\partial\overline{V}\partial.\\
\end{eqnarray*}
\end{enumerate}
\end{proof}
\smallskip
Let
\begin{eqnarray*}
P_{*}^{\overline{\Phi}}(\overline{G})=\{\sum_i a_i\alpha_i\in P_*(\overline{G})\mid {\overline{\Phi}(\sum_i a_i\alpha_i)=\sum_i a_i\alpha_i}, a_i\in R\}.
\end{eqnarray*}
By Theorem~\ref{th-6.3} (iv), the boundary operator $\partial$ maps  $P_{n}^{\overline{\Phi}}(\overline{G})$ to $P_{n-1}^{\overline{\Phi}}(\overline{G})$. Thus
$\{P_{*}^{\overline{\Phi}}(\overline{G}),\partial_*\}$ is a sub-chain complex of $\{P_*(\overline{G}),\partial_*\}$ consisting of all $\overline{\Phi}$-invariant chains,  called the Morse complex. Moreover, by Theorem~\ref{th-6.3} and \cite[Theorem~7.2]{forman1},  we have that
\begin{eqnarray*}
\overline{\Phi}^{N}=\overline{\Phi}^{N+1}=\cdots=\overline{\Phi}^{\infty}
\end{eqnarray*}
for some sufficiently large positive integer  $N$, where  $\overline{\Phi}^{\infty}=\lim\limits_{N \to \infty}\overline{\Phi}^{N}$.

To give the discrete Morse theory for digraphs, we first prove a quasi-isomorphism of chain complexes.

\begin{theorem}\label{th-6.11}
Suppose $\Omega_*(G)$ is $\overline{V}$-invariant (that is, $\overline{V}(\Omega_n(G))\subseteq \Omega_{n+1}(G)$ for each $n\geq 0$).  There is a quasi-isomorphism
\begin{eqnarray*}
\Omega_*(G)\longrightarrow \Omega_*(G)\cap P_*^{\overline{\Phi}}(\overline{G}).
\end{eqnarray*}
\end{theorem}
\begin{proof}
By the proof of \cite[Theorem~7.3]{forman1}, we have the following  chain homotopy
\begin{eqnarray}
\overline{\Phi}^{\infty}: P_*(\overline{G})&&\longrightarrow P_*^{\overline{\Phi}}(\overline{G}); \label{eq-s2}\\
\iota: P_*^{\overline{\Phi}}(\overline{G})&&\longrightarrow P_*(\overline{G})\nonumber.
\end{eqnarray}
Here $\iota$ is the canonical inclusion. It is proved in \cite[Theorem~7.3]{forman1} that
\begin{eqnarray}\label{eq-s11}
\overline{\Phi}^{\infty}\circ \iota=\mathrm{Id}
\end{eqnarray}
and
\begin{eqnarray}\label{eq-s112}
\iota\circ \overline{\Phi}^{\infty}\simeq\mathrm{Id}.
\end{eqnarray}

Firstly, we will prove that
\begin{eqnarray}\label{eq-s1}
\overline{\Phi}^{\infty}\mid_{\Omega_*(G)}:  \Omega_*(G)\longrightarrow \Omega_*(G)\cap P_*^{\overline{\Phi}}(\overline{G})
\end{eqnarray}
is well-defined.

For any $x=\sum {a_i \alpha_i}\in \Omega_n(G)\subseteq P_n(\overline{G})$ where $a_i\in R$ and $\alpha_i$ are allowed elementary $n$-paths on $G$. Since $\Omega_*(G)$ is $\overline{V}$-invariant, then $\overline{\Phi}(x)\in \Omega_n(G)$. Hence
\begin{eqnarray*}
\overline{\Phi}^{\infty}(x)\in \Omega_n(G).
\end{eqnarray*}
On the other hand, by  (\ref{eq-s2}),
\begin{eqnarray*}
\overline{\Phi}^{\infty}(x)\in P^{\overline{\Phi}}_*(\overline{G}).
\end{eqnarray*}
Hence (\ref{eq-s1}) is well-defined.

Secondly, by (\ref{eq-s11}) and (\ref{eq-s112}), we have
\begin{eqnarray*}
(\overline{\Phi}^{\infty}\mid_{\Omega_*(G)})\circ (\iota\mid_{\Omega_*(G)\cap P_*^{\overline{\Phi}}(\overline{G})})=\mathrm{Id}.
\end{eqnarray*}
It follows that
\begin{eqnarray*}
(\overline{\Phi}^{\infty}\mid_{\Omega_*(G)})_*\circ (\iota\mid_{\Omega_*(G)\cap P_*^{\overline{\Phi}}(\overline{G})})_*=\mathrm{Id}.
\end{eqnarray*}
and $(\overline{\Phi}^{\infty}\mid_{\Omega_*(G)})_*$  is surjective. Here $(\overline{\Phi}^{\infty}\mid_{\Omega_*(G)})_*$ and $(\iota\mid_{\Omega_*(G)\cap P_*^{\overline{\Phi}}(\overline{G})})_*$ are homomorphisms between homology groups $H_*(\Omega_*(G))$ and $H_*(\Omega_*(G)\cap P_*^{\overline{\Phi}}(\overline{G}))$ induced by morphisms between chain complexes $\Omega_*(G)$ and $\Omega_*(G)\cap P_*^{\overline{\Phi}}(\overline{G})$. It leaves us to show $(\overline{\Phi}^{\infty}\mid_{\Omega_*(G)})_*$ is injective. That is,  for any element $x\in \mathrm{Ker}\partial\mid_{\Omega_n(G)}$, if
\begin{eqnarray*}
(\overline{\Phi}^{\infty}\mid_{\Omega_*(G)})_*(x)
\end{eqnarray*}
is a boundary in $\Omega_{n+1}(G)\cap P_{n+1}^{\overline{\Phi}}(\overline{G})$, then $x$ is a boundary in $\Omega_{n+1}(G)$.

Suppose there exists an element $y\in \Omega_{n+1}(G)\cap P_{n+1}^{\overline{\Phi}}(\overline{G})$ such that $\partial y=\overline{\Phi}^{\infty}(x)$. Since
\begin{eqnarray*}
&&\overline{\Phi}^{\infty}(x)=\overline{\Phi}^{N}(x)\\
&&~~~~~~~~~=(\mathrm{Id}+\partial\overline{V})^{N}(x)\\
&&~~~~~~~~~=\left(C_{N}^{0}(\mathrm{Id})^{N}+C_{N}^{1}\mathrm{Id}^{(N-1)}\partial\overline{V}+C_{N}^{2}\mathrm{Id}^{(N-2)}(\partial\overline{V})^2+\cdots +C_{N}^{N}(\partial\overline{V})^{N}\right)(x)\\
&&~~~~~~~~~=x+\left(C_{N}^{1}\partial\overline{V}+C_{N}^{2}(\partial\overline{V})^2+\cdots +C_{N}^{N}(\partial\overline{V})^{N}\right)(x)\\
&&~~~~~~~~~=x+\partial\left(C_{N}^{1}\overline{V}+C_{N}^{2}(\overline{V}\partial\overline{V})+\cdots +C_{N}^{N}(\overline{V}\partial\overline{V}\cdots\partial\overline{V})\right)(x)
\end{eqnarray*}
and $\Omega_*(G)$ is $\overline{V}$-invariant, then $L(x)\in \Omega_{n+1}(G)$ where
\begin{eqnarray*}
L=C_{N}^{1}\overline{V}+C_{N}^{2}(\overline{V}\partial\overline{V})+\cdots +C_{N}^{N}(\overline{V}\partial\overline{V}\cdots\partial\overline{V}).
\end{eqnarray*}
Hence
\begin{eqnarray*}
\partial y-\partial L(x)=x
\end{eqnarray*}
which implies $x=\partial(y-L(x))$ and $x$ is a boundary in $\Omega_{n+1}(G)$.

The theorem is proved.
\end{proof}

Denote
\begin{eqnarray*}
\Omega_*(G)\cap P_*^{\overline{\Phi}}(\overline{G})=\Omega_*^{\overline{\Phi}}(G).
\end{eqnarray*}
By Theorem~\ref{th-6.11}, we have the discrete Morse theory for digraphs as follows.
\begin{corollary}\label{co-111}
Let $G$ be a digraph and $f$ a discrete Morse function on $G$ satisfying Condition (*). Let $\overline{f}$ be the extension of $f$ on $\overline{G}$ and $\overline{V}=\mathrm{grad}\overline{f}$ the discrete gradient vector field on $\overline{G}$. Suppose $\Omega_*(G)$ is $\overline{V}$-invariant. Then
\begin{eqnarray*}
H_m(G)\cong H_m(\Omega_*^{\overline{\Phi}}(G)), m \geq 0.
\end{eqnarray*}
\end{corollary}

Furthermore, for each $n\geq 0$,  let $\mathrm{Crit}_n(G)$ be the free $R$-module consisting of all the formal linear combinations of critical $n$-paths on $G$.  Then $\mathrm{Crit}_n(G)$ is a   sub-$R$-module of $P_n(G)$.  By \cite[Theorem~8.2]{forman1}, there is an isomorphism of graded $R$-modules
\begin{eqnarray*}\label{eq-8.2}
\overline{\Phi}^{\infty}\mid_{\mathrm{Crit}_*(\overline{G})}: \mathrm{Crit}_*(\overline{G})\longrightarrow P_*^{\overline{\Phi}}(\overline{G}).
\end{eqnarray*}
Hence, by Corollary~\ref{co-111}, we have that
\begin{theorem}\label{th-8.3}
Let $G$ be a digraph and $f$ a discrete Morse function on $G$ satisfying Condition (*). Let $\overline{f}$ be the extension of $f$ on $\overline{G}$ and $\overline{V}=\mathrm{grad}\overline{f}$ the discrete gradient vector field on $\overline{G}$. Suppose $\Omega_*(G)$ is $\overline{V}$-invariant. Then
\begin{eqnarray}\label{eq-2.17}
H_m(G)\cong H_m(\{\Omega_n(G)\cap\overline{\Phi}^{\infty}(\mathrm{Crit}_n(\overline{G})), \partial_n\}_{n\geq 0}).
\end{eqnarray}
\end{theorem}

\begin{example}\label{ex-2.17}

Consider the following digraph $G$ and its transitive closure $\overline{G}$. Then
\begin{eqnarray*}
\Omega_*(G)=R(v_0, v_1, v_2, v_3, v_0v_1, v_1v_2, v_2v_3, v_0v_3)
\end{eqnarray*}

Let $f: V(G)\longrightarrow [0, +\infty)$ be  a function on $G$ with $f(v_0)=0$ and $f(v_i)>0, 0<i\leq 3$. Then $f$ is  a discrete Morse function on $G$ satisfying Condition ($*$). By Theorem~\ref{th-666}, $f$ can be extended to be  a Morse function $\overline{f}$ on $\overline{G}$ such that $\overline{f}(v) = f(v)$ for all vertices $v\in V(\overline{G})$.
\begin{figure}[!htbp]
 \begin{center}
\begin{tikzpicture}[line width=1.5pt]

\coordinate [label=left:$v_0$]    (A) at (2,2);
 \coordinate [label=left:$v_1$]   (B) at (2,4);
 \coordinate  [label=right:$v_2$]   (C) at (4,4);
\coordinate  [label=right:$v_3$]   (D) at (4,2);
\fill (2,2) circle (2.5pt);
\fill (2,4) circle (2.5pt);
\fill (4,4) circle (2.5pt);
\fill (4,2) circle (2.5pt);

\coordinate[label=left:$G$:] (G) at (1,3);
\draw [line width=1.5pt]  (A) -- (B);
\draw [line width=1.5pt]  (B) -- (C);
\draw [line width=1.5pt]   (C) -- (D);
\draw [line width=1.5pt]  (A) -- (D);
 \draw[->] (2,2) -- (2,3);

\draw[->] (2,4) -- (3,4);

\draw[->] (4,4) -- (4,3);

\draw[->] (2,2) -- (3,2);

\coordinate [label=left:$v_0$]    (A) at (2+6,2);
 \coordinate [label=left:$v_1$]   (B) at (2+6,4);
 \coordinate  [label=right:$v_2$]   (C) at (4+6,4);
\coordinate  [label=right:$v_3$]   (D) at (4+6,2);
\fill (2+6,2) circle (2.5pt);
\fill (2+6,4) circle (2.5pt);
\fill (4+6,4) circle (2.5pt);
\fill (4+6,2) circle (2.5pt);

\coordinate[label=left:$\overline{G}$:] (H) at (7,3);
\draw [line width=1.5pt]  (A) -- (B);
\draw [line width=1.5pt]  (A) -- (C);
\draw [line width=1.5pt]  (B) -- (D);
\draw [line width=1.5pt]  (B) -- (C);
\draw [line width=1.5pt]   (C) -- (D);
\draw [line width=1.5pt]  (A) -- (D);
 \draw[->] (2+6,2) -- (2+6,3);
 \draw[->] (2+6,2) -- (9.5,3.5);

\draw[->] (2+6,4) -- (3+6,4);
\draw[->] (2+6,4) -- (9.5,2.5);

\draw[->] (4+6,4) -- (4+6,3);

\draw[->] (2+6,2) -- (3+6,2);

 \end{tikzpicture}
\end{center}
\caption{Example~\ref{ex-2.17}.}
\end{figure}


Since $\overline{f}(v_0)=0$ and $v_0\to v_i$ ($i\not=0$) are directed edges on $\overline{G}$, all allowed elementary paths on $\overline{G}$ except for the $0$-path $\{v_0\}$ are not critical. Hence,
\begin{eqnarray*}\label{eq-e33}
\mathrm{Crit}_*(\overline{G})=R(v_0).
\end{eqnarray*}

Let $\overline{V}=\mathrm{grad}\overline{f}$ be the discrete gradient vector field on $\overline{G}$. Then
\begin{eqnarray}
&& \overline{V}(v_1)=-v_0v_1, ~~~~~\overline{V}(v_2)=-v_0v_2,~~~~~\overline{V}(v_3)=-v_0v_3,\nonumber\\
&& \overline{V}(v_1v_3)=-v_0v_1v_3,\overline{V}(v_2v_3)=-v_0v_2v_3,\overline{V}(v_1v_2)=-v_0v_1v_2 \label{eq-t3}\\
&& \overline{V}(\alpha)=0 ~ \text{for any other allowed elementary path}~\alpha \text{ on}~ \overline{G},\nonumber
\end{eqnarray}
in which (\ref{eq-t3}) implies that $\Omega_*(G)$  is not  $\overline{V}-\mathrm{invariant}$.

Let $\overline \Phi=\mathrm{Id}+ \partial \overline V+\overline V\partial$  be the discrete gradient flow of $\overline{G}$. Then
\begin{eqnarray*}
&&\overline \Phi(v_0)=v_0,~~~~\overline \Phi(v_1)=v_0,\\
&&\overline \Phi(v_2)=v_0,~~~~\overline \Phi(v_3)=v_0,\\
&&\overline \Phi(\alpha)=0 ~ \text{for any other allowed elementary path}~\alpha \text{ on}~ \overline{G}.
\end{eqnarray*}
By calculate directly, we have that $\overline \Phi^\infty=\overline \Phi$. Then
\begin{eqnarray*}
&&\overline{\Phi}^{\infty}(\mathrm{Crit}_*(\overline{G}))=R(v_0),\\
&&\Omega_*(G)\cap\overline{\Phi}^{\infty}(\mathrm{Crit}_*(\overline{G}))=R(v_0).
\end{eqnarray*}
Hence,
\begin{eqnarray*}
&&H_0(\Omega_*(G)\cap\overline{\Phi}^{\infty}(\mathrm{Crit}_*(\overline{G})))\cong  R,\\
&&H_m(\Omega_*(G)\cap\overline{\Phi}^{\infty}(\mathrm{Crit}_*(\overline{G})))=0~\text{for } m>0.
\end{eqnarray*}

By \cite[Proposition~4.7]{9}, $H_1(G)\cong R$ and by \cite[Theorem~4.6]{9}, $H_1(\overline{G})=0$.
\begin{eqnarray*}
&&H_1(G)\not=H_1(\Omega_*(G)\cap\overline{\Phi}^{\infty}(\mathrm{Crit}_*(\overline{G}))),\\
&&H_1(G)\not=H_1(\overline{G}).
\end{eqnarray*}
\end{example}

\begin{remark}\label{re-s3}
If the condition that  $\Omega_*(G)$ is  $\overline{V}-\text{invariant}$ in  Theorem~\ref{th-8.3} does not hold  for some digraphs, there may be no isomorphism of homology groups given in (\ref{eq-2.17}).
\end{remark}

\begin{remark}\label{re-2.17}
In general, the homology groups of a digraph $G$ and its transitive closure $\overline{G}$ are different (in the sense of isomorphism).
\end{remark}

\section{Digraphs Satisfying Condition ($*$)}

In this section, we  will show that there are quite general digraphs on which Morse functions satisfying Condition ($*$) can be defined, and give some examples to illustrate Theorem~\ref{th-8.3}.

Inspired by  Lemma~\ref{le-1.14}, we give the following theorem.

\begin{theorem}\label{th-33}
Let $G$ be a digraph and  $f: V(G)\longrightarrow [0, +\infty)$ a function on $G$ given by

\begin{eqnarray*}\label{eq-17.2}
f(v)=\left\{
\begin{array}{cc}
0, & \text{ if  } v=v', \\
\not=0, & \text{ if } v\not=v'.
\end{array}
\right.
\end{eqnarray*}
Here $v'\in V(G)$ does not belong to any directed loop on $G$. Then $f$ is  a discrete Morse function on $G$ satisfying Condition ($*$).
\end{theorem}
\begin{proof}
Let  $\alpha=v_0\cdots v_n$ be an arbitrary allowed elementary path on $G$. There are two cases.

{\sc Case~1}.  $v_i\not=v'$ for any $0\leq i\leq n$. We assert that there exists only one index $k$ with $-1\leq k\leq {n}$ such that $\gamma'=v_0\cdots v_k v' v_{k+1}\cdots v_n$ (for $k=-1$, $\gamma'=v'v_0\cdots v_n$) is an allowed elementary path on $G$. Suppose to the contrary, $\gamma''=v_0\cdots v_j v' v_{j+1}\cdots v_n$ is another allowed elementary path on $G$, $j\not=k$. Without loss of generality, $k<j$. Then $\tilde{\gamma}=v'v_{k+1}\cdots v_j v'$ is a directed loop on $G$ which contradicts that $v'$ does not belong to any directed loop. Hence for any allowed elementary path $\gamma>\alpha$,
\begin{eqnarray*}
\#\Big\{\gamma>\alpha\mid f(\gamma)=f(\alpha)\Big\}\leq 1
\end{eqnarray*}
and for any allowed elementary path $\beta<\alpha$, $f(\beta)<f(\alpha)$. 

{\sc Case~2}. $v_i=v'$ for some $0\leq i\leq n$. We assert that there is no any other vertex $v_j = v'$ for $0\leq j\not= i \leq n$. Suppose to the contrary, $v_i = v_j =v'$, $i\not=j$. Without loss of generality, $i<j$. Since $\alpha$ is an allowed elementary path on $G$, $v_{k-1}\not=v_k$ for each $1\leq k\leq n$. Then $j\not=i+1$.  Let $\alpha'=v_i\cdots v_j$. Then $\alpha'$ is a directed loop on $G$ which contradicts that $v'$ does not belong to any directed loop.  Hence, for any allowed elementary path $\beta<\alpha$,
\begin{eqnarray*}
\#\Big\{\beta<\alpha\mid f(\beta)=f(\alpha)\Big\}\leq 1
\end{eqnarray*}
and for any allowed elementary path $\gamma>\alpha$, $f(\gamma)> f(\alpha)$.

Combining Case 1 and Case 2,  by Definition~\ref{def-1.1},  $f$ is a discrete Morse function on $G$.  Moreover, since there is  only one vertex $v'\in V(G)$ such that $f(v')=0$,  $f$ satisfies   Condition ($*$). The theorem follows.
\end{proof}

Finally, for illustrating  the application of discrete Morse theory in the calculation of simplified homology groups of digraphs, we give the following examples.
\begin{example}\label{ex-1}
Let $G$ be a square. Then the transitive closure of $G$ is  a digraph $\overline{G}$  with $V(\overline{G})=V(G)$ and
\begin{eqnarray*}
E(\overline{G})=E(G)\cup \{v_0\to v_3\}.
\end{eqnarray*}

Let $f$ be a function on $G$ such that
\begin{eqnarray*}
f(v_0)=1, f(v_1)=0, f(v_2)=2, f(v_3)=3.
\end{eqnarray*}

\begin{figure}[!htbp]
 \begin{center}
\begin{tikzpicture}[line width=1.5pt]

\coordinate [label=left:$v_0$]    (A) at (1,0);
 \coordinate [label=right:$v_1$]   (B) at (2.5,1);
 \coordinate  [label=right:$v_2$]   (C) at (2.5,-1);
\coordinate  [label=right:$v_3$]   (D) at (4,0);
\fill (1,0) circle (2.5pt);
\fill (2.5,1) circle (2.5pt);
\fill (2.5,-1) circle (2.5pt);
\fill (4,0) circle (2.5pt);

\coordinate[label=left:$G$:] (G) at (0.5,0);
\draw [line width=1.5pt]  (A) -- (B);
\draw [line width=1.5pt]  (A) -- (C);
\draw [line width=1.5pt]   (B) -- (D);
\draw [line width=1.5pt]  (C) -- (D);
 \draw[->] (1,0) -- (2,2/3);

\draw[->] (2.5,1) -- (3,2/3);

\draw[->] (1,0) -- (2,-2/3);

\draw[->] (2.5,-1) -- (3,-2/3);

\coordinate [label=left:$v_0$]    (A) at (1+6,0);
 \coordinate [label=right:$v_1$]   (B) at (2.5+6,1);
 \coordinate  [label=right:$v_2$]   (C) at (2.5+6,-1);
\coordinate  [label=right:$v_3$]   (D) at (4+6,0);
\fill (1+6,0) circle (2.5pt);
\fill (2+6.5,1) circle (2.5pt);
\fill (2.5+6,-1) circle (2.5pt);
\fill (4+6,0) circle (2.5pt);

\coordinate[label=left:$\overline{G}$:] (H) at (6+0.5,0);
\draw [line width=1.5pt]  (A) -- (B);
 \draw [line width=1.5pt]  (A) -- (C);
 \draw [line width=1.5pt]  (A) -- (D);
\draw [line width=1.5pt] (B) -- (D);
\draw [line width=1.5pt]  (C) -- (D);
 \draw[->] (7,0) -- (8,2/3);

\draw[->]  (8.5,1) -- (9,2/3);

\draw[->]  (7,0) -- (8,-2/3);
\draw[->] (8.5,-1) -- (9,-2/3);
 \draw[->] (7,0) -- (9,0);

 \end{tikzpicture}
\end{center}

\caption{Example~\ref{ex-1}.}
\end{figure}
By  Theorem~\ref{th-33}, $f$ is a discrete Morse function on $G$ satisfying Condition ($*$) and by Theorem~\ref{th-666}, $f$ can be extended to be  a Morse function $\overline{f}$ on $\overline{G}$ such that $\overline{f}(v) = f(v)$ for all vertices $v\in V(\overline{G})$.
Then
\begin{eqnarray*}
\mathrm{Crit}_*(G)&=&R(v_1, v_2, v_0v_2, v_2v_3, v_0v_1v_3, v_0v_2v_3),\\
\mathrm{Crit}_*(\overline{G})&=&R(v_1, v_2, v_0v_2, v_2v_3,  v_0v_2v_3),\\
\Omega_*(G)&=&R(v_0, v_1, v_2, v_3, v_0v_1, v_0v_2, v_1v_3, v_2v_3, v_0v_1v_3-v_0v_2v_3).
\end{eqnarray*}
Note that $\mathrm{Crit}_*(\overline{G})\cap P(G)\not=\mathrm{Crit}_*(G)$.

Let $\overline{V}=\mathrm{grad}\overline{f}$ be the discrete gradient vector field on $\overline{G}$. Then
\begin{eqnarray*}
&& \overline{V}(v_0)=v_0v_1, ~~~\overline{V}(v_3)=-v_1v_3,\\
&& \overline{V}(v_0v_3)=v_0v_1v_3,\\
&& \overline{V}(\alpha)=0 ~ \text{for any other allowed elementary path}~\alpha \text{ on}~ \overline{G},\\
&& \overline{V}(\Omega_n(G))\subseteq \Omega_{n+1}(G)~ \text{for any}~n\geq 0. ~That~ is, \Omega_*(G) \text{ is } \overline{V}-\mathrm{invariant}.
\end{eqnarray*}

Let $\overline \Phi=\mathrm{Id}+ \partial \overline V+\overline V\partial$  be the discrete gradient flow of $\overline{G}$. Then
\begin{eqnarray*}
&\overline \Phi(v_0)=v_1,&\overline \Phi(v_1)=v_1,\\
&\overline \Phi(v_2)=v_2,
&\overline \Phi(v_3)=v_1,\\
&\overline \Phi(v_0v_1)=0,
&\overline \Phi(v_0v_2)=v_0v_2-v_0v_1,\\
&\overline \Phi(v_1v_3)=0,
&\overline \Phi(v_2v_3)=v_2v_3-v_1v_3,\\
&\overline \Phi(v_0v_3)=0,
&\overline \Phi(v_0v_1v_3)=0,\\
&~~~~~~~~~~~~~~~~~~~~\overline \Phi(v_0v_2v_3)=v_0v_2v_3-v_0v_1v_3.
\end{eqnarray*}
By calculate directly, we have that $\overline \Phi^\infty=\overline \Phi$. Then
\begin{eqnarray*}
\overline{\Phi}^{\infty}(\mathrm{Crit}_*(\overline{G}))&=&R(v_1, v_2, v_0v_2-v_0v_1, v_2v_3-v_1v_3, v_0v_2v_3-v_0v_1v_3),\\
\Omega_*(G)\cap\overline{\Phi}^{\infty}(\mathrm{Crit}_*(\overline{G}))&=&R(v_1, v_2, v_0v_2-v_0v_1, v_2v_3-v_1v_3, v_0v_2v_3-v_0v_1v_3).
\end{eqnarray*}

Therefore,
\begin{eqnarray*}
&&\partial_1(v_0v_2-v_0v_1)=v_2-v_1,~~~ \partial_1(v_2v_3-v_1v_3)=v_1-v_2\\
&&\partial_2(v_0v_2v_3-v_0v_1v_3)=(v_0v_2-v_0v_1)+(v_2v_3-v_1v_3)
\end{eqnarray*}
and
\begin{eqnarray*}
H_0(\Omega_*(G)\cap\overline \Phi^\infty(\mathrm{Crit}_*(\bar{G})))&=&R,\\
 H_1(\Omega_*(G)\cap\overline \Phi^\infty(\mathrm{Crit}_*(\bar{G})))&=&0,\\
 H_m(\Omega_*(G)\cap\overline \Phi^\infty(\mathrm{Crit}_*(\bar{G})))&=&0 ~\text{for} ~m\geq 2,
\end{eqnarray*}
which are consistent with the  path homology groups of $G$ given in \cite[Proposition~4.7]{9}.
\end{example}

\smallskip


\begin{example}\label{ex-2}
Consider the digraph $G$ given in Example~\ref{ex-1}.  Let $f: V(G)\longrightarrow [0, +\infty)$ be  a function on $G$ with $f(v_0)=0$ and $f(v_i)>0, 0<i\leq 3$, which is different from the function given in Example~\ref{ex-1}. By  Theorem~\ref{th-33}, $f$ is  a discrete Morse function on $G$ satisfying Condition ($*$) and by Theorem~\ref{th-666}, $f$ can be extended to be  a Morse function $\overline{f}$ on $\overline{G}$ such that $\overline{f}(v) = f(v)$ for all vertices $v\in V(\overline{G})$.  Obviously,
\begin{eqnarray*}
\Omega_*(G)=R(v_0, v_1, v_2, v_3, v_0v_1, v_0v_2, v_1v_3, v_2v_3, v_0v_1v_3-v_0v_2v_3)
\end{eqnarray*}

Since $\overline{f}(v_0)=0$ and $v_0\to v_i$ ($i\not=0$) are directed edges on $\overline{G}$, it follows that all allowed elementary paths on $\overline{G}$ except for the $0$-path $\{v_0\}$ are not critical. Hence $\mathrm{Crit}_*(\overline{G})=R(v_0)$.

Let $\overline{V}=\mathrm{grad}\overline{f}$ be the discrete gradient vector field on $\overline{G}$. Then
\begin{eqnarray}
&& \overline{V}(v_1)=-v_0v_1, ~~~\overline{V}(v_2)=-v_0v_2,~~~\overline{V}(v_3)=-v_0v_3,\nonumber\\
&& \overline{V}(v_1v_3)=-v_0v_1v_3,~~~~~~~~~~~~~~~\overline{V}(v_2v_3)=-v_0v_2v_3, \label{eq-t1}\\
&& \overline{V}(\alpha)=0 ~ \text{for any other allowed elementary path}~\alpha \text{ on}~ \overline{G},\nonumber
\end{eqnarray}
in which (\ref{eq-t1}) implies that $\Omega_*(G)$  is not  $\overline{V}-\mathrm{invariant}$.

Let $\overline \Phi=\mathrm{Id}+ \partial \overline V+\overline V\partial$  be the discrete gradient flow of $\overline{G}$. Then
\begin{eqnarray*}
&&\overline \Phi(v_0)=v_0,~~~~~\overline \Phi(v_1)=v_0,\\
&&\overline \Phi(v_2)=v_0,~~~~~\overline \Phi(v_3)=v_0,\\
&&\overline \Phi(\alpha)=0 ~ \text{for any other allowed elementary path}~\alpha \text{ on}~ \overline{G}.
\end{eqnarray*}
By calculate directly, we have that $\overline \Phi^\infty=\overline \Phi$. Then
\begin{eqnarray*}
\overline{\Phi}^{\infty}(\mathrm{Crit}_*(\overline{G}))&=&R(v_0),\\
\Omega_*(G)\cap\overline{\Phi}^{\infty}(\mathrm{Crit}_*(\overline{G}))&=&R(v_0).
\end{eqnarray*}

Therefore,
\begin{eqnarray*}
H_0(G)&=&H_0(\Omega_*(G)\cap\overline{\Phi}^{\infty}(\mathrm{Crit}_*(\overline{G})))\cong R,\\
H_m(G)&=&H_m(\Omega_*(G)\cap\overline{\Phi}^{\infty}(\mathrm{Crit}_*(\overline{G})))=0~\text{for all } m\geq 1.
\end{eqnarray*}
\end{example}

\smallskip
\begin{remark}
By Example~\ref{ex-1} and Example~\ref{ex-2}, we know  that in the discrete Morse theory for digraphs, the selection of zero points of discrete Morse functions is very important to simplify the calculation of homology groups. Generally speaking, we can choose the vertex with larger degree in the transitive closure of a digraph as the zero point.
\end{remark}

\smallskip

\begin{example}\label{ex-22}
Consider the following digraph $G$ and its transitive closure $\overline{G}$. Let $f: V(G)\longrightarrow [0, +\infty)$ be  a function on $G$ with $f(v_0)=0$ and $f(v_i)>0$, $0<i \leq 5$.

\begin{figure}[!htbp]
 \begin{center}
\begin{tikzpicture}[line width=1.5pt]

\coordinate [label=left:$v_0$]    (A) at (3,1);
 \coordinate [label=left:$v_1$]   (B) at (2,2);
 \coordinate  [label=right:$v_2$]   (C) at (4,2);
\coordinate  [label=left:$v_3$]   (D) at (2,4);
\coordinate  [label=right:$v_4$]   (E) at (4,4);
 \coordinate [label=left:$v_5$]   (F) at (3,5);

\coordinate[label=left:$G$:] (H) at (1,3);
   \draw [line width=1.5pt]  (A) -- (B);
    \draw [line width=1.5pt]  (A) -- (C);
      \draw [line width=1.5pt]  (B) -- (D);
       \draw [line width=1.5pt]  (B) -- (E);
            \draw [line width=1.5pt]  (C) -- (D);
                 \draw [line width=1.5pt]  (C) -- (E);
                  \draw [line width=1.5pt]  (F) -- (D);
 \draw [line width=1.5pt]  (F) -- (E);
 \draw[->] (3,1) -- (2.5,1.5);
 \draw[->] (3,1) -- (3.5,1.5);

\draw[->] (2,2) -- (2,3.5);
\draw[->] (2,2) -- (3.8,3.8);

\draw[->] (4,2) -- (4,3.5);
\draw[->] (4,2) -- (2.2,3.8);

\draw[->] (3,5) -- (2.5,4.5);
\draw[->] (3,5) -- (3.5,4.5);

\fill (3,1)  circle (2.5pt) (2,2)  circle (2.5pt)  (4,2) circle (2.5pt) (2,4) circle (2.5pt) (4,4) circle (2.5pt) (3,5)  circle (2.5pt);

\coordinate [label=left:$v_0$]    (A) at (3+6,1);
 \coordinate [label=left:$v_1$]   (B) at (2+6,2);
 \coordinate  [label=right:$v_2$]   (C) at (4+6,2);
\coordinate  [label=left:$v_3$]   (D) at (2+6,4);
\coordinate  [label=right:$v_4$]   (E) at (4+6,4);
 \coordinate [label=left:$v_5$]   (F) at (3+6,5);

\coordinate[label=left:$\overline{G}$:] (I) at (1+6,3);
   \draw [line width=1.5pt]  (A) -- (B);
    \draw [line width=1.5pt]  (A) -- (C);
    \draw [line width=1.5pt]  (A) -- (D);
    \draw [line width=1.5pt]  (A) -- (E);
      \draw [line width=1.5pt]  (B) -- (D);
       \draw [line width=1.5pt]  (B) -- (E);
            \draw [line width=1.5pt]  (C) -- (D);
                 \draw [line width=1.5pt]  (C) -- (E);
                  \draw [line width=1.5pt]  (F) -- (D);
 \draw [line width=1.5pt]  (F) -- (E);
 \draw[->] (3+6,1) -- (2.5+6,1.5);
 \draw[->] (3+6,1) -- (3.5+6,1.5);
 \draw[->] (3+6,1) -- (8.4,2.8);
\draw[->] (3+6,1) -- (9.6,2.8);

\draw[->] (2+6,2) -- (2+6,3.5);
\draw[->] (2+6,2) -- (3.8+6,3.8);

\draw[->] (4+6,2) -- (4+6,3.5);
\draw[->] (4+6,2) -- (2+6.2,3.8);

\draw[->] (3+6,5) -- (2.5+6,4.5);
\draw[->] (3+6,5) -- (3.5+6,4.5);

\fill (3+6,1)  circle (2.5pt) (2+6,2)  circle (2.5pt)  (4+6,2) circle (2.5pt) (2+6,4) circle (2.5pt) (4+6,4) circle (2.5pt) (3+6,5)  circle (2.5pt);

 \end{tikzpicture}
\end{center}
\caption{Example~\ref{ex-22}.}
\end{figure}
 By  Theorem~\ref{th-33}, $f$ is  a discrete Morse function on $G$ satisfying Condition ($*$). By Theorem~\ref{th-666}, $f$ can be extended to be a Morse function $\overline{f}$ on $\overline{G}$ such that $\overline{f}(v)=f(v) $ for all vertices $v\in V(\overline{G})$. Then
\begin{eqnarray*}
\mathrm{Crit}_*(\overline{G})&=&R(v_0, v_5, v_5v_3, v_5v_4),\\
\Omega_*(G)&=&R(v_0, v_1, v_2, v_3, v_4, v_5,  v_0v_1, v_0v_2, v_1v_3, v_1v_4,  v_2v_3, v_2v_4, v_5v_3, v_5v_4,\\
 &&v_0v_1v_3-v_0v_2v_3, v_0v_1v_4-v_0v_2v_4).\\
\end{eqnarray*}

Let $\overline{V}=\mathrm{grad}\overline{f}$ be the discrete gradient vector field on $\overline{G}$. Then
\begin{eqnarray}
&& \overline{V}(v_1)=-v_0v_1, ~~~\overline{V}(v_2)=-v_0v_2,\nonumber\\
&&\overline{V}(v_3)=-v_0v_3,~~~\overline{V}(v_4)=-v_0v_4,\nonumber\\
&&\overline{V}(v_1v_3)=-v_0v_1v_3,~~~\overline{V}(v_1v_4)=-v_0v_1v_4,~~~\label{eq-s1l}\\
&& \overline{V}(v_2v_3)=-v_0v_2v_3,~~~\overline{V}(v_2v_4)=-v_0v_2v_4,~~~\label{eq-s12}\\
&& \overline{V}(\alpha)=0 ~ \text{for any other allowed elementary path}~\alpha \text{ on}~ \overline{G}.\nonumber
\end{eqnarray}
By (\ref{eq-s1l}) and (\ref{eq-s12}),  $\overline{V}(\Omega_1(G))\nsubseteq \Omega_2(G)$. This implies that $\Omega_*(G)$ is not $\overline{V}-\mathrm{invariant}$.

Let $\overline \Phi=\mathrm{Id}+ \partial \overline V+\overline V\partial$  be the discrete gradient flow of $\overline{G}$. Then
\begin{eqnarray*}
&\overline \Phi(v_0)=v_0,&\overline \Phi(v_1)=v_0,\\
&\overline \Phi(v_2)=v_0,
&\overline \Phi(v_3)=v_0,\\
&\overline \Phi(v_4)=v_0,&\overline \Phi(v_5)=v_5,\\
&\overline \Phi(v_0v_1)=0,&\overline \Phi(v_0v_2)=0,\\
&\overline \Phi(v_0v_3)=0,
&\overline \Phi(v_0v_4)=0,\\
&\overline \Phi(v_1v_3)=0,
&\overline \Phi(v_1v_4)=0,\\
&\overline \Phi(v_2v_3)=0,&\overline \Phi(v_2v_4)=0,\\
\end{eqnarray*}
\begin{eqnarray*}
&\overline \Phi(v_5v_3)=v_5v_3-v_0v_3,
&\overline \Phi(v_5v_4)=v_5v_4-v_0v_4,\\
&\overline \Phi(v_0v_1v_4)=0,
&\overline \Phi(v_0v_1v_3)=0,\\
&\overline \Phi(v_0v_2v_4)=0,
&\overline \Phi(v_0v_2v_3)=0.\\
\end{eqnarray*}
By calculate directly, we have that $\overline \Phi^\infty=\overline \Phi$. Then
\begin{eqnarray*}
\overline \Phi^\infty(\mathrm{Crit}_*(\bar{G}))&=&R(v_0, v_5, v_5v_3-v_0v_3, v_5v_4-v_0v_4),\\
\Omega_*(G)\cap\overline \Phi^\infty(\mathrm{Crit}_*(\bar{G}))&=&R(v_0, v_5, v_5v_3-v_0v_3, v_5v_4-v_0v_4).
\end{eqnarray*}


Hence,
\begin{eqnarray*}
\partial_1(v_5v_3-v_0v_3)&=&v_0-v_5\\
\partial_1(v_5v_4-v_0v_4)&=&v_0-v_5
\end{eqnarray*}
and
\begin{eqnarray*}
H_0(\Omega_*(G)\cap\overline \Phi^\infty(\mathrm{Crit}_*(\bar{G})))&=&R\\
 H_1(\Omega_*(G)\cap\overline \Phi^\infty(\mathrm{Crit}_*(\bar{G})))&=&R\\
 H_m(\Omega_*(G)\cap\overline \Phi^\infty(\mathrm{Crit}_*(\bar{G})))&=&0~ \text{for}~m\geq 2,
\end{eqnarray*}
which are consistent with the  path homology groups of $G$.
\end{example}

\begin{remark}
By Example~\ref{ex-2} and Example~\ref{ex-22}, the condition that  $\Omega_*(G)$ is  $\overline{V}-\mathrm{invariant}$ in  Theorem~\ref{th-8.3} is sufficient but not necessary. That is, even if the digraph does not satisfy this condition, we may have an isomorphism of homology groups given in (\ref{eq-2.17}).
\end{remark}

\smallskip
\section{Further Discussion}
In this section, we will study the  matrix representation of Theorem~\ref{th-8.3}, which will be helpful for us to find  efficient algorithms computing the homology (persistent homology) groups of digraphs applying our discrete Morse theory for digraphs in the future.

Let $G$ be a digraph and $\overline{G}$ the transitive closure of $G$. Choose all allowed elementary $n$-paths as a basis of $P(\overline{G})$, denoted as $B_n$. Let $M(\bullet)$ be the matrix corresponding to the operator $\bullet$ and $E_n$ the  identity matrix of order $n$. Let
\begin{eqnarray*}
\overline{V}_n&:& P_n(\bar{G})\longrightarrow P_{n+1}(\bar{G}),\\
\partial_n&:& P_n(\bar{G})\longrightarrow P_{n-1}(\bar{G}),\\
\overline{\Phi}_n&:& P_n(\bar{G})\longrightarrow P_n(\bar{G}),\\
\overline \Phi_n^\infty\mid_{\mathrm{Crit}_n(\bar{G})}&:& \mathrm{Crit}_n(\bar{G})\longrightarrow P_n^{\overline \Phi}(\bar{G}).
\end{eqnarray*}
for each $n\geq 0$.


We illustrate the calculation process of homology groups in Example~\ref{ex-1} with the matrix representation of operators. Since
\begin{eqnarray*}
P_0(\bar{G})&=&R(v_0, v_1, v_2, v_3);\\
P_1(\bar{G})&=&R(v_0v_1, v_0v_2, v_0v_3, v_1v_3, v_2v_3);\\
P_2(\bar{G})&=&R(v_0v_1v_3, v_0v_2v_3);\\
P_3(\bar{G})&=&0.
\end{eqnarray*}
Then
\begin{equation}\label{eq-01}
\overline{V}_1
\left[
\begin{array}{c}
   v_0v_1 \\
   v_0v_2 \\
   v_0v_3\\
   v_1v_3\\
   v_2v_3
\end{array}
\right]
=
\left[
\begin{array}{cc}
    0 & 0 \\
    0 & 0 \\
    1 & 0 \\
    0 & 0\\
    0 & 0
\end{array}
\right]
\left[
\begin{array}{c}
    v_0v_1v_3\\
    v_0v_2v_3
\end{array}
\right],
\end{equation}

\begin{equation}\label{eq-02}
\overline{V}_0
\left[
\begin{array}{c}
   v_0\\
   v_1 \\
   v_2 \\
   v_3
\end{array}
\right]
=
\left[
\begin{array}{ccccc}
 1& 0 & 0 & 0 & 0 \\
 0& 0 & 0 & 0 & 0 \\
 0& 0 & 0 & 0 & 0 \\
 0& 0 & 0 & -1 & 0
\end{array}
\right]
\left[
\begin{array}{c}
   v_0v_1 \\
   v_0v_2 \\
   v_0v_3\\
   v_1v_3\\
   v_2v_3
\end{array}
\right],
\end{equation}

\begin{equation}\label{eq-03}
\partial_2
\left[
\begin{array}{c}
   v_0v_1v_3\\
   v_0v_2v_3
\end{array}
\right]
=
\left[
\begin{array}{ccccc}
 1 & 0 & -1 & 1 &  0\\
 0 & 1 & -1 & 0 &  1
\end{array}
\right]
\left[
\begin{array}{c}
    v_0v_1\\
    v_0v_2\\
    v_0v_3\\
    v_1v_3\\
    v_2v_3
\end{array}
\right],
\end{equation}
and
\begin{equation}\label{eq-04}
\partial_1
\left[
\begin{array}{c}
   v_0v_1\\
    v_0v_2\\
    v_0v_3\\
    v_1v_3\\
    v_2v_3
\end{array}
\right]
=
\left[
\begin{array}{cccc}
 -1 &  1 &  0 &  0\\
 -1 &  0 &  1 &  0 \\
 -1 &  0 &  0 &  1 \\
  0 & -1 &  0 &  1 \\
  0 &  0 & -1 &  1
\end{array}
\right]
\left[
\begin{array}{c}
    v_0\\
    v_1\\
    v_2\\
    v_3
\end{array}
\right].
\end{equation}

By (\ref{eq-01})-(\ref{eq-04}), we have that
\begin{equation*}
M(\overline{V}_1)=
 \left[
 \begin{array}{ccccc}
     0 & 0 \\
     0 & 0 \\
     1 & 0 \\
     0 & 0 \\
     0 & 0
 \end{array}
 \right],
 \end{equation*}

 \begin{equation*}
M(\overline{V}_0)=
 \left[
 \begin{array}{ccccc}
   1 & 0 & 0 & 0 & 0 \\
   0 & 0 & 0 & 0 & 0 \\
   0 & 0 & 0 & 0 & 0 \\
   0 & 0 & 0 & -1 & 0
 \end{array}
 \right],
 \end{equation*}

  \begin{equation*}
M(\partial_2)=
 \left[
 \begin{array}{ccccc}
   1 & 0 & -1 & 1 & 0 \\
   0 & 1 & -1 & 0 & 1\\
 \end{array}
 \right],
 \end{equation*}
 and
 \begin{equation*}
M(\partial_1)=
 \left[
 \begin{array}{cccc}
-1 & 1 & 0 & 0\\
-1 & 0 & 1 & 0\\
-1 & 0 & 0 & 1\\
0 & -1 & 0 & 1\\
0 & 0 & -1 & 1\\
 \end{array}
 \right].
 \end{equation*}
 Hence,
  \begin{equation*}
M(\overline \Phi_1)=E_1+M(\partial_1)M(\overline{V}_0)+M(\overline{V}_1)M(\partial_2)=
 \left[
 \begin{array}{ccccc}
0 & 0 & 0 & 0 & 0\\
-1 & 1 & 0 & 0 & 0\\
0 & 0 & 0 & 0 & 0\\
0 & 0 & 0 & 0 & 0\\
0 & 0 & 0 & -1 & 1
 \end{array}
 \right].
 \end{equation*}
By calculate directly, we have that $(M(\overline \Phi_1))^{\infty}=M(\overline \Phi_1)\cdot M(\overline \Phi_1)\cdot \cdots=M(\overline \Phi_1)=M(\overline \Phi_1^\infty)$. Since $\mathrm{Crit}_1(\bar{G})=R(v_0v_2, v_2v_3)$. Then $M(\overline \Phi_1^\infty\mid_{\mathrm{Crit}_1(\bar{G})})$ is the matrix composed of the second and the fifth rows of $M(\overline \Phi_1)$. That is,
  \begin{equation*}\label{eq-09}
M(\overline \Phi_1^\infty\mid_{\mathrm{Crit}_1(\bar{G})})=
 \left[
 \begin{array}{ccccc}
-1 & 1 & 0 & 0 & 0\\
0 & 0 & 0 & -1 & 1
 \end{array}
 \right].
 \end{equation*}
Since $\Omega_1(G)=R(v_0v_1, v_0v_2, v_1v_3, v_2v_3)$, it follows that
\begin{eqnarray*}
\Omega_1(G)\cap\overline \Phi^\infty(\mathrm{Crit}_1(\bar{G}))=R(v_0v_2-v_0v_1, v_2v_3-v_1v_3).
\end{eqnarray*}

By (\ref{eq-04}), we have that
\begin{equation}\label{eq-10}
M(\partial_1\mid_{\Omega_1(G)\cap\overline \Phi^\infty(\mathrm{Crit}_1(\bar{G}))})=
 \left[
 \begin{array}{cccc}
0 & -1 & 1 & 0\\
0 & 1 & -1 & 0\\
 \end{array}
 \right].
 \end{equation}
 Hence, we can obtain that
 \begin{eqnarray*}
 \mathrm{Ker}(\partial_1\mid_{\Omega_1(G)\cap\overline \Phi^\infty(\mathrm{Crit}_1(\bar{G}))})=R(v_0v_2-v_0v_1+v_2v_3-v_1v_3)
 \end{eqnarray*}
 and
 \begin{eqnarray*}
 \mathrm{Im}(\partial_1\mid_{\Omega_1(G)\cap\overline \Phi^\infty(\mathrm{Crit}_1(\bar{G}))})=R(v_1-v_2)
 \end{eqnarray*}
 from (\ref{eq-10}).

Similarly,  the matrix of $\overline \Phi_0: P_0(\bar{G})\longrightarrow  P_0(\bar{G})$ is
 \begin{equation*}
M(\overline \Phi_0)=
 \left[
 \begin{array}{cccc}
0& 1& 0 & 0 \\
0& 1& 0 & 0\\
0& 0& 1 & 0\\
0& 1& 0 & 0
 \end{array}
 \right]
 \end{equation*}
and the matrix of $\overline \Phi_2: P_2(\bar{G})\longrightarrow  P_2(\bar{G})$ is
 \begin{equation*}
M(\overline \Phi_2)=
 \left[
 \begin{array}{cc}
0 & 0 \\
-1& 1
 \end{array}
 \right].
 \end{equation*}

By calculate directly, \begin{eqnarray*}
(M(\overline \Phi_0))^{\infty}=M(\overline \Phi_0)=M(\overline \Phi_0^\infty)
\end{eqnarray*}
 and
 \begin{eqnarray*}
 (M(\overline \Phi_2))^{\infty}=M(\overline \Phi_2)=M(\overline \Phi_2^\infty).
 \end{eqnarray*}

Then
\begin{eqnarray*}
\overline \Phi_0^\infty(\mathrm{Crit}_0(\bar{G}))&=&\overline \Phi_0(\mathrm{Crit}_0(\bar{G}))\\
&=&\mathrm{Crit}_0(\bar{G}).
\end{eqnarray*}
Hence,
\begin{eqnarray*}
\Omega_0(G)\cap\overline \Phi^\infty(\mathrm{Crit}_0(\bar{G}))&=&\mathrm{Crit}_0(\bar{G})\\
&=&R(v_1, v_2)
\end{eqnarray*}
and
 \begin{eqnarray*}
 \mathrm{Ker}(\partial_0\mid_{\Omega_0(G)\cap\overline \Phi^\infty(\mathrm{Crit}_0(\bar{G}))})=R(v_1, v_2).
 \end{eqnarray*}

Furthermore, since $\Omega_2(G)=R(v_0v_1v_3-v_0v_2v_3)$ and $\mathrm{Crit}_2(\bar{G})=R(v_0v_2v_3)$, it follows that
\begin{eqnarray*}
\Omega_2(G)\cap\overline \Phi^\infty(\mathrm{Crit}_2(\bar{G}))=R(v_0v_1v_3-v_0v_2v_3).
\end{eqnarray*}
By (\ref{eq-03}), we have that
\begin{equation}\label{eq-11}
M(\partial_2\mid_{\Omega_2(G)\cap\overline \Phi^\infty(\mathrm{Crit}_2(\bar{G}))})=
 \left[
 \begin{array}{ccccc}
1 & -1 & 0 & 1 &-1
 \end{array}
 \right].
 \end{equation}
Hence, we  can obtain that
 \begin{eqnarray*}
 \mathrm{Ker}(\partial_2\mid_{\Omega_2(G)\cap\overline \Phi^\infty(\mathrm{Crit}_2(\bar{G}))})=0
 \end{eqnarray*}
 and
 \begin{eqnarray*}
 \mathrm{Im}(\partial_2\mid_{\Omega_2(G)\cap\overline \Phi^\infty(\mathrm{Crit}_2(\bar{G}))})=R(v_0v_1-v_0v_2+v_1v_3-v_2v_3)
 \end{eqnarray*}
from (\ref{eq-11}).

Therefore,
\begin{eqnarray*}
H_0(\Omega_*(G)\cap\overline \Phi^\infty(\mathrm{Crit}_*(\bar{G})))&=&R,\\
H_1(\Omega_*(G)\cap\overline \Phi^\infty(\mathrm{Crit}_*(\bar{G})))&=&0,\\
H_m(\Omega_*(G)\cap\overline \Phi^\infty(\mathrm{Crit}_*(\bar{G})))&=&0 ~\text{for} ~m\geq 2,
\end{eqnarray*}



\bigskip
\noindent {\bf Acknowledgement}. The authors would like to express deep gratitude to the reviewer(s) for their careful reading, valuable comments, and helpful suggestions.

\noindent {$^{*}$ Partially supported by National Natural Science Foundation of China, Grant No. 12071245.}

\noindent {$^{\dag}$ Partially supported by Science and Technology Project of Hebei Education Department (QN2019333), the Natural Fund of Cangzhou Science and Technology Bureau (No.197000002) and a Project of Cangzhou Normal University (No.xnjjl1902).}

\noindent {$^{\S}$ Partially supported by DMS-1737873.}

 Yong Lin

 Address: Yau Mathematical Sciences Center, Tsinghua University, Beijing 100084, China.

 e-mail: yonglin@tsinghua.edu.cn

\medskip

Chong Wang  (for correspondence)

 Address: $^1$School of Mathematics, Renmin University of China, Beijing 100872, China.

 $^2$School of Mathematics and Statistics, Cangzhou Normal University, 061000 China.

 e-mail:  wangchong\_618@163.com

  \medskip

Shing-Tung Yau

Department of Mathematics, Harvard University, Cambridge MA 02138, USA.

e-mail: yau@math.harvard.edu
\end{document}